\documentclass[reqno,12pt,centertags]{amsart}
\usepackage{amsmath,amsthm,amscd,amssymb,latexsym,upref,stmaryrd}
\usepackage[cp1251]{inputenc}
\usepackage[english]{babel}
\usepackage{mathtext}
\usepackage{amsfonts}
\usepackage{graphicx}
\usepackage[numbers,sort&compress]{natbib}
\usepackage{hyperref}
\newcommand*{\mailto}[1]{\href{mailto:#1}{\nolinkurl{#1}}}
\newcommand{\lb}{\label}
\usepackage{amsmath}
\usepackage{amsthm}
\newtheorem{corollary}{Corollary}
\newtheorem{theorem}{Theorem}
\newtheorem{lemma}{Lemma}
\newtheorem{remark}{Remark}
\theoremstyle{proposition}

\begin{document}

\thispagestyle{empty}

\noindent{\large\bf Inverse spectral problems for the Sturm-Liouville operator with discontinuity}
\\

\noindent {\bf  Xiao-Chuan Xu}\footnote{Department of Applied
Mathematics, School of Science, Nanjing University of Science and Technology, Nanjing, 210094, Jiangsu,
People's Republic of China, {\it Email:
xiaochuanxu@126.com}}
{\bf and Chuan-Fu Yang}\footnote{Department of Applied
Mathematics, School of Science, Nanjing University of Science and Technology, Nanjing, 210094, Jiangsu,
People's Republic of China, {\it Email: chuanfuyang@njust.edu.cn}}
\\

\noindent{\bf Abstract.}
{In this work, we consider the Sturm-Liouville operator on a finite interval $[0,1]$ with discontinuous conditions at $1/2$. We prove that if the potential is known a priori on a subinterval $[b,1]$ with $b\ge1/2$, then parts of two spectra can uniquely determine the potential and all parameters in discontinuous conditions and boundary conditions. For the case $b<1/2$, parts of either one or two spectra can uniquely determine the potential and a part of parameters.}

\medskip
\noindent {\it Keywords:} {Sturm-Liouville operator; Discontinuous condition;
Inverse spectral problem; Mixed data}

\medskip
\noindent {\it 2010 Mathematics Subject Classification:} {34B24, 47E05}

\section{Introduction}

We consider the following Sturm-Liouville boundary value problem
\begin{equation}\label{fc1}
-y''(x)+q(x)y(x)=\lambda y(x), \ 0<x<1,
\end{equation}
\begin{equation}\label{btj}
y'(0)-hy(0)=0,\quad y'(1)+Hy(1)=0,
\end{equation}
with the discontinuous conditions
\begin{equation}\label{blxtj}
y\left(\frac{1}{2}\!+\!0\right)\!=\!a_1y\left(\frac{1}{2}\!-\!0\right),\
y'\left(\frac{1}{2}\!+\!0\right)\!=\!a_1^{-1}y'\left(\frac{1}{2}\!-\!0\right)\!+\!a_2y\left(\frac{1}{2}\!-\!0\right).
\end{equation}
Here $\lambda$ is the spectral parameter,
$q(\cdot), h, H, a_1, a_2$ are real, $q(\cdot)\in L^2[0, 1]$, and $a_1>0$.
Denote the boundary value problem (\ref{fc1}), (\ref{btj}) and
(\ref{blxtj}) by $B=B(q,h,H,a_1,a_2)$.

The boundary value problems with a discontinuous point inside the interval arise in
mathematics, mechanics, radio electronics, geophysics, and other
fields of science and technology. Such problems are connected
with discontinuous material properties (see, for example,
\cite{AND,HAL,KKW,KRU}).

The problem $B$ has been studied by many scholars (see \cite{AM,FYU2,KRU,SYU,CWI,YYA1,YYA2,YYA3,YYA,YUR1} and the references therein).
In general, for recovering the potential function on the whole
interval and all parameters in discontinuity conditions and boundary conditions, it is necessary to specify two spectra of the problem $B$ with different boundary conditions (see \cite{RS,YUR1}). We are interested in recovering the potential and all parameters in discontinuous conditions and boundary conditions from parts of two spectra provided the potential is known a priori on a subinterval. This is the so-called inverse spectral problem with mixed given data, which has been considered by some scholars (see, for example, \cite{HAL,SYU,YYA1}). Specifically, the authors of \cite{SYU} assumed that $a_1,a_2$ and $H$ are given, and   proved that if $q(x)$ is known on $[b, 1]$ with $b<1/2$, then less than one spectra can uniquely determine $h$ and $q(x)$ on $[0, b]$, and if $b=1/2$ then it needs to specify the whole one spectra. The paper \cite{YYA1} dealt with the inverse problem by using Gesztesy-Simon's method under the assumption that $q(x)$ is known on more than half of the whole interval, and gave a uniqueness theorem. We also note that inverse problems with mixed given data for  differential operators were studied by many authors (see, for example, \cite{BU,BC,FYU3,FYU4,GS1,GS2,MH,MH1,OKC,RAM,MPI,RFB,RB,SYU,YYA,YYA3,YUR2}).

In this paper, we study the inverse spectral problems with mixed given data for the problem $B$ under the assumption that $q(x)$ is known on $[b, 1]$ with $b\in(0,1]$. The main method is partly based on ideas in \cite{RFB,RAM}, which require asymptotics of the
eigenvalues and eigenfunctions, and some techniques of complex analysis.

\section{Main Results}
 Denote $B_\infty=B(q,h,\infty,a_1,a_2)$, which means that $y'(1)+Hy(1)=0$ is replaced by $y(1)=0$ in (\ref{btj}). Note that the operators $B$ and $B_\infty$ are self-adjoint. Let $\Lambda_1:=\{\lambda_n\}_{n\in \mathbb{N}_0}$ and $\Lambda_2:=\{\mu_n\}_{n\in \mathbb{N}_0}$ be the spectra of the problems $B$ and $B_\infty$, respectively, where $\mathbb{N}_0:=\mathbb{N}\cup\{0\}$. It is well known that the sequence $\{\lambda_n\}_{n\in \mathbb{N}_0}$ and $\{\mu_n\}_{n\in \mathbb{N}_0}$ satisfy
the following asymptotics \cite{SYU,YYA2,YUR1}
\begin{equation}\label{lbdn}
\sqrt{\lambda_n}=n\pi+\frac{\omega+(-1)^n\omega_1}{n\pi}+o\left(\frac{1}{n}\right),
\end{equation}
and
\begin{equation}\label{xc}
\sqrt{\mu_n}=\gamma_n+\frac{\omega_0}{n\pi}+o\left(\frac{1}{n}\right),\quad \gamma_n:=\left(n+\frac{1}{2}\right)\pi+(-1)^n\arcsin a,
\end{equation}
respectively. Here
\begin{equation}\label{xc1}
 \left\{\begin{split}
         &a=\frac{a_1-a_1^{-1}}{a_1+a_1^{-1}},\\
          &\omega=\omega_0+H,\quad \omega_0=\frac{a_2}{a_1+a_1^{-1}}+\frac{1}{2}\int_0^1q(t)dt+h,\\
         & \omega_1=\frac{a_2}{a_1+a_1^{-1}}+a\left[\int_{\frac{1}{2}}^1q(t)dt+H-\left(\frac{1}{2}\int_0^1q(t)dt+h\right)\right].
        \end{split}\right.
\end{equation}

Denote $x_+:=\max\{x,0\}$ and for the given real sequence $\{x_n\}_{n\in \mathbb{N}_0}:=A$, define a counting function
\begin{equation*}\label{cx3}
  N_A(r):={\rm{\# }}\{n\in\mathbb{N}_0:x_n\le r^2,\ r>0\}.
\end{equation*}
Let $S_i$ $(i=1,2)$ be the subsets of the spectral sets $\Lambda_i$, respectively. Assume that the subscripts of the elements in $S_1$ satisfy the condition \emph{(I): including infinitely many even and odd numbers}.

Now we state the main results of this article.
\begin{theorem}\lb{2.0}
Assume that $q(x)$ is known a priori on a.e. $[b,1]$ with $b\in[\frac{1}{2},1]$, and $S_1$ satisfies the condition (I).
If the spectral subsets $S_i$ $(i=1,2)$ satisfy

{(i)}
for some $r_0>0$ and arbitrary $r\ge r_0$ there holds
\begin{equation}\label{xc2}
N_{S_1\cup S_2}(r)\ge{\sigma} N_{\Lambda_1\cup\Lambda_2}(r)-m(r),\quad \sigma>b,
\end{equation}
where $m(r)=o(r)$ as $r\to\infty$; or

{(ii)}
there exist positive constants $\sigma_i$ such that
 \begin{equation}\label{cx0}
\!\!\! \sum_{\kappa_{1,n}\in S_1,\ n\ge0 }\!\!\!\!\!\!\!\!\frac{(\kappa_{1,n}-n^2\pi^2/\sigma_1^2)_+}{n^2}+ \!\!\!\!\!\!\!\! \sum_{\kappa_{2,n}\in S_2,\ n\ge0 }\!\!\!\!\!\!\!\frac{(\kappa_{2,n}-\gamma_n^2/\sigma_2^2)_+}{n^2}<\infty,\quad \sigma_1+\sigma_2=2b,
 \end{equation}
then $S_1\cup S_2$ uniquely determines $q(x)$ a.e. on $[0,b]$ and $h,H,a_1,a_2$.
\end{theorem}

\begin{corollary}\lb{2.1}
Assume that $q(x)$ on $[\frac{1}{2},1]$, $a_1$ and $H$ are known a prior, then $\Lambda_1$ uniquely determines $q(x)$ a.e. on $[0,\frac{1}{2}]$, $h$ and $a_2$.
\end{corollary}
\begin{remark}
One may regularly choose sequences $\{\lambda_{m_1(n)}\}_{n\in\mathbb{N}_0}= S_1$ and $\{\mu_{m_2(n)}\}_{n\in\mathbb{N}_0}= S_2$ with $m_i(n)$ $(i=1,2)$ satisfying
\begin{equation}\label{vcx}
    m_1(n)=\frac{n}{\sigma_1}(1+\epsilon_{1,n}),\
  m_2(n)=\frac{\gamma_n}{\sigma_2}(1+\epsilon_{2,n}),\ \epsilon_{i,n}\to 0,\ n\to\infty.
\end{equation}
Note that the choice (\ref{vcx}) is a particular case of that in Theorem \ref{2.0}, which was used in some earlier works.
\end{remark}
\begin{remark}
If $H$ is given then the condition (I) in Theorem \ref{2.0} can be removed. In addition, one restricts $h<\infty$ in Theorem \ref{2.0},
actually the method used in this paper can also be applicable in the case $h=\infty$ (i.e., $y(0)=0$ in (\ref{btj})).
\end{remark}
\begin{remark}
From the proofs of Theorem \ref{2.0} and Corollary \ref{2.1} one can find that if $b\in (0,\frac{1}{2})$ the results in Theorem \ref{2.0} also hold provided that $a_1,a_2$ and $H$ are given. That is to say, if $q(x)$ is known on $[b,1]$ with $b\in (0,\frac{1}{2}),$ $a_1,a_2$ and $H$ are given, then $q(x)$ and $h$ can be uniquely determined by parts of either one or two spectra which satisfy the same conditions as those in Theorem \ref{2.0}.
\end{remark}

Theorem \ref{2.0} can be generalized to the case for parts of more than two spectra. For convenience, denote $H_1:=H,\;H_2:=\infty$. For the fixed $N\ge3$, denote $B_i:=B(q,h,H_i,a_1,a_2)$, where $H_i(\ne H)\in \mathbb{R}$ and $H_i\ne H_j$ if $i\ne j$. Denote $\gamma_{i,n}:=n\pi$ for $i=\overline{1,N}\setminus\{2\}$ and $\gamma_{2,n}:=\gamma_n,n\in\mathbb{N}_0$.

 Let $S_i$ ($i=\overline{1,N}$) be the subsets of the sets $\Lambda_i$, respectively, where $\Lambda_i$ denote the spectral sets of the problem $B_i$ ($i=\overline{1,N}$). Denote $S:=\cup_{i=1}^NS_i$ and $\Lambda:=\cup_{i=1}^N\Lambda_i$.

The generalization of Theorem \ref{2.0} is as follows.
\begin{theorem}\lb{3.0}
Assume that $q(x)$ is known a priori on a.e. $[b,1]$ with $b\in[\frac{1}{2},1]$, and $S_i$ $(i=\overline{1,N}\setminus\{2\})$ satisfy the condition (I).
If the sets $S_i $ $(i=\overline{1,N})$ satisfy that

(i) for some $r_0>0$ and arbitrary $r\ge r_0$ there holds
\begin{equation*}
N_S(r)\ge{\sigma} N_{\Lambda}(r)-m(r),\quad \sigma>\frac{2b}{N},
\end{equation*}
where $m(r)=o(r)$ as $r\to\infty$;
or

(ii)
there exist positive constants $\sigma_i$ such that
 \begin{equation*}
  \sum_{i=1}^N\sum_{\kappa_{i,n}\in S_i,\ n\ge0}\!\!\!\!\!\!\frac{(\kappa_{i,n}-\gamma_{i,n}^2/\sigma_i^2)_+}{n^2}<\infty,\quad \displaystyle\sum_{i=1}^N\sigma_i=2b,
 \end{equation*}
then $S$ uniquely determines $q(x)$ a.e. on $[0,b]$ and $h,\;H_i,\;a_1,\;a_2.$
\end{theorem}

\section{Preliminaries}
In this section, we provide some preliminaries for proving the main results.

Together with the problem $B$ we consider a boundary value problem
$\tilde{B}=B(\tilde{q},\tilde{h},\tilde{H},\tilde{a}_1,\tilde{a}_2)$ of the same form but with
different coefficients $\tilde{q},\tilde{h},\tilde{H},\tilde{a}_1$ and $\tilde{a}_2$. We agree that if a certain
symbol $\delta$ denotes an object related to $B$, then
$\tilde{\delta}$ will denote an analogous object related to
$\tilde{B}$.

Let us recall the product of eigenfunctions \cite{HAL,CWI}.
Let $y(x,\lambda)$ be the solution of the equation (\ref{fc1})
satisfying the initial conditions $y(0)=1,\ y'(0)=h$ and the
conditions (\ref{blxtj}). It is well known that $y(x,\lambda)$ is an entire function of $\lambda$ of the order $\frac{1}{2}$. Let $k=\sqrt{\lambda}$,
there exists a bounded function $K(x,t)$
such that
\begin{equation}\label{yyw1}
y(x,\lambda)\tilde{y}(x,\lambda)=\frac{1}{2}+\frac{1}{2}\cos(2kx)+\frac{1}{2}\int_0^xK(x,t)\cos(2kt)dt
\end{equation}
for $0\leq x<\frac{1}{2}$, and
\begin{equation}\label{yyw2}
\begin{array}{rl}
y(x,\lambda)\tilde{y}(x,\lambda)\!\!\!\!&=A_1(\lambda)
+A_2\cos(2kx)+A_3\cos2k(x-\frac{1}{2})\\
&\quad+A_4\cos2k(x-1)+\frac{1}{2}\int_0^xK(x,t)\cos(2kt)dt
\end{array}
\end{equation}
for $\frac{1}{2}<x\leq 1$ and the given parameter $a_1$, where
\begin{equation*}
\left\{\!\!\!\!\!\!\!
\begin{array}{rl}
&A_1(\lambda)=\frac{1}{2a_1^2}+\frac{a_1^2-a_1^{-2}}{2}y(\frac{1}{2}-0,\lambda)
\tilde{y}(\frac{1}{2}-0,\lambda),\\
&A_2=\frac{(a_1+a_1^{-1})^2}{8},\\
&A_3=\frac{a_1^2-a_1^{-2}}{4},\\
&A_4=\frac{(a_1-a_1^{-1})^2}{8}.
\end{array}\right.
\end{equation*}

Under the assumption $a_1=\tilde{a}_1$, one can easily obtain from (\ref{blxtj}) that
\begin{equation}\lb{yz}
(\tilde{y}'y-\tilde{y}y')(\textstyle{\frac{1}{2}+0},\lambda)
=(\tilde{y}'y-\tilde{y}y')(\textstyle{\frac{1}{2}-0},\lambda)+a_1(\tilde{a}_2-a_2)(\tilde{y}y)\left(\frac{1}{2}-0,\lambda\right).
\end{equation}

Let $\sqrt{\lambda}=\sigma+i\tau$. Using the inequality $|\cos
\sqrt{\lambda}x|<e^{|\tau|x}$, we obtain
\begin{equation}\label{yyw11}
|y(x,\lambda)\tilde{y}(x,\lambda)|\leq M_1e^{2b|\tau|}
\end{equation}
for $0\leq x\leq b\leq\frac{1}{2}$, and
\begin{equation}\label{yyw21}
\begin{array}{rl}
|y(x,\lambda)\tilde{y}(x,\lambda)|\!\!\!\!&\leq
M_1e^{|\tau|}+A_2e^{2x|\tau|}+|A_3|e^{(2x-1)|\tau|}+|A_3|e^{|\tau|}\\
&\quad+
|A_4|e^{(2-2x)|\tau|}+M_3e^{2x|\tau|}\\
&\leq M_4e^{2b|\tau|},
\end{array}
\end{equation}
for $\frac{1}{2}\leq x\leq b\leq 1$, where $M_i\; (i=\overline{1,4})$ are some positive constants.

\begin{lemma}\lb{l1}
For $ b\in [0,1]$ and $ Q(\cdot)\in L^2[0,b]$, if there are constants $M_5$ and $M_6$ such that
\begin{equation}\label{xc6}
  \int_0^b\!Q(x)y(x,\lambda)\tilde{y}(x,\lambda)dx+\!M_5y(\textstyle{1 \over 2}-0,\lambda)\tilde{y}(\textstyle{1 \over 2}-0,\lambda)+M_6=0,\quad\forall\lambda\ge0,
\end{equation}
then $Q(x)=0$ a.e. on $[0,b]$, and $M_5=M_6=0$.
\end{lemma}

\begin{proof}
The proof is partly from \cite{YYA}.
Firstly, we discuss the case $b\in[0,\frac{1}{2}]$.
Substituting (\ref{yyw1}) into (\ref{xc6}), we obtain that for all
$k\ge0$,
\begin{equation*}
\begin{array}{l}
\int_0^bQ(x)\left[1+\cos(2kx)\right]dx
+\int_0^bQ(x)\left[\int_0^xK(x,t)\cos(2kt)dt\right]dx\\
+\frac{M_5}{2}[1+\cos k+\int_0^{\frac{1}{2}}K(\frac{1}{2}-0,t)\cos(2kt)dt]+M_6=0,
\end{array}
\end{equation*}
which can be rewritten as
\begin{equation}\lb{Qjfs10}
\begin{array}{l}
\int_0^bQ(x)dx+\int_0^b\cos(2kt)\left[Q(t)
+\int_t^bQ(x)K(x,t)dx\right]dt\\
+\frac{M_5}{2}[1+\cos k+\int_0^{\frac{1}{2}}K(\frac{1}{2}-0,t)\cos(2kt)dt]+M_6=0.
\end{array}
\end{equation}

Letting $k\to+\infty$ in (\ref{Qjfs10}) and observing that the limit of $\cos k$ does't exist for $k\to +\infty$, and using Riemann-Lebesgue lemma we
obtain that
\begin{equation}\lb{xc7}
M_5=0,\quad \int_0^bQ(x)dx+M_6=0,
\end{equation}
and hence
\begin{equation*}
\int_0^b\cos(2kt)\left[ Q(t)+\int_t^bQ(x)K(x,t)dx\right]dt=0.
\end{equation*}
Since the function system $\{\cos(2kt)\}_{k\ge0}$ is complete in $L^2[0,b]$, then
\begin{equation*}
Q(t)+\int_t^bQ(x)K(x,t)dx=0,\ 0<t<b.
\end{equation*}
This equation is a homogeneous Volterra integral equation which
has only the zero solution. Thus $Q(x)=0$ a.e. on $[0,b]$, and $M_6=0$ from Eq.(\ref{xc7}).

Secondly, we consider the case $b\in(1/2,1]$.
Substituting (\ref{yyw1}) and (\ref{yyw2}) into (\ref{xc6}), we obtain, for all $k\ge0$,
\begin{equation}\label{QK1}
\begin{array}{l}
\int_0^{\frac{1}{2}}Q(x)[1+\cos(2kx)+\int_0^xK(x,t)\cos(2kt)dt]dx\\
+\int_{\frac{1}{2}}^bQ(x)[
2A_2\cos(2kx)+2A_3\cos2k(x-\frac{1}{2})
+2A_4\cos2k(x-1)\\+\int_0^xK(x,t)\cos(2kt)dt]dx
+[2M_5+(a_1^2-a_1^{-2})\int_{\frac{1}{2}}^bQ(x)dx]\\ \times y(\textstyle{1 \over 2}-0,\lambda)\tilde{y}(\textstyle{1 \over 2}-0,\lambda)+2M_6+a_1^{-2}\int_{\frac{1}{2}}^bQ(x)dx=0.
\end{array}
\end{equation}

Letting $k\rightarrow+\infty$ in (\ref{QK1}) and observing that the limit of $y(\textstyle{1 \over 2}-0,\lambda)\tilde{y}(\textstyle{1 \over 2}-0,\lambda)$ does't exist for $k\to +\infty$, we see
from Riemann-Lebesgue lemma that
\begin{equation}\lb{xc8}
\left\{\begin{array}{l}
\int_0^{\frac{1}{2}}Q(x)dx+2M_6+a_1^{-2}\!\!\int_{\frac{1}{2}}^bQ(x)dx=0,\\ 2M_5+(a_1^2-a_1^{-2}\!)\!\!\int_{\frac{1}{2}}^bQ(x)dx=0,
\end{array}\right.
\end{equation}
and hence
\begin{equation}\label{QK2}
\begin{array}{l}
\int_0^{\frac{1}{2}}Q(x)[\cos(2kx)+\int_0^xK(x,t)\cos(2kt)dt]dx\\
\qquad+\int_{\frac{1}{2}}^bQ(x)[2A_2\cos(2kx)+2A_3\cos2k(x-\frac{1}{2})\\
\qquad\quad+2A_4\cos2k(x-1)+\int_0^xK(x,t)\cos(2kt)dt]dx=0.
\end{array}
\end{equation}

We will change variables to obtain an equation of the form
\begin{equation*}\lb{F}
\int_0^b\left[F(t)+\int_t^bK(x,t)Q(x)dx\right]\cos(2kt)dt=0,
\end{equation*}
which implies
\begin{equation}\lb{F1}
F(t)+\int_t^bK(x,t)Q(x)dx=0 \mbox{\ \ for\ \ }t\in (0,b),
\end{equation}
since the function system $\{\cos(2kt)\}_{k\ge0}$ is complete in $L^2[0,b]$.
The form of $F(t)$ will alow us to conclude that $Q(x)=0$ a.e. on $[0,b]$. If we can prove it, then it follows from (\ref{xc8}) that $M_5=0$ and $M_6=0$.

We first consider the terms with $K(x,t)$ in (\ref{QK2}). Since $K(x,t)$ is bounded on $(x,t)\in\![0,1]\!\times\![0,1]$
and $Q(x)$ is integrable on $[0,1]$, by Fubini's theorem
\begin{equation}\label{QK3}
\begin{split}
\int_0^{\frac{1}{2}}&Q(x)\!\int_0^xK(x,t)\cos(2kt)dtdx\!+\!\int_{\frac{1}{2}}^bQ(x)\!\int_0^x\!K(x,t)\cos(2kt)dtdx\\
&=\int_0^b\int_t^bK(x,t)Q(x)dx\cos(2kt)dt.
\end{split}
\end{equation}
We next consider the remaining terms in (\ref{QK2}). Specifically
we have
\begin{equation}\label{Q1}
\int_0^{\frac{1}{2}}\!Q(x)\cos(2kx)dx\!+\!\int_{\frac{1}{2}}^b\!2A_2Q(x)\!\cos(2kx)dx\!=\!\int_0^b\widehat{Q}(t)\cos(2kt)dt,
\end{equation}
where
\begin{equation*}
\widehat{Q}(t)=\left\{\begin{array}{ll} Q(t)&\mbox{ for } t\in
[0,\frac{1}{2}],\\
2A_2Q(t)&\mbox{ for } t\in (\frac{1}{2},b],\end{array}\right.
\end{equation*}

\begin{equation}\label{Q2}
\begin{array}{l}
\int_{\frac{1}{2}}^b2A_3Q(x)\cos2k(x-\frac{1}{2})dx=\int_0^{b-\frac{1}{2}}2A_3Q(t+\frac{1}{2})\cos(2kt)dt,
\end{array}
\end{equation}
and
\begin{equation}\label{Q3}
\begin{array}{l}
\int_{\frac{1}{2}}^b2A_4Q(x)\cos2k(x-1)dx=\int_{1-b}^{\frac{1}{2}}2A_4Q(1-t)\cos(2kt)dt.
\end{array}
\end{equation}

Together with Eqs.(\ref{QK3})$-$(\ref{Q3}), the form of $F(t)$ in
(\ref{F1}) is as follows.

If $1\ge b\geq \frac{3}{4}$,
\begin{equation}\label{Fxs1}
F(t)\!\!=\!\!\left\{\begin{array}{ll} Q(t)+2A_3Q(t+\frac{1}{2}),&  t\in
[0,1-b],\\
Q(t)\!\!+\!\!2A_3Q(t+\frac{1}{2})\!\!+\!\!2A_4Q(1-t),&
t\in
[1-b,b-\frac{1}{2}],\\
Q(t)+2A_4Q(1-t),& t\in
[b-\frac{1}{2},\frac{1}{2}],\\
2A_2Q(t),&  t\in [\frac{1}{2},b];
\end{array}\right.
\end{equation}

If $\frac{1}{2}<b<\frac{3}{4}$,
\begin{equation}\label{Fxs2}
F(t)\!\!=\!\!\left\{\begin{array}{ll} Q(t)+2A_3Q(t+\frac{1}{2}), &
 t\in
[0,b-\frac{1}{2}]\\
Q(t),&  t\in
[b-\frac{1}{2},1-b]\\
Q(t)+2A_4Q(1-t),&  t\in
[1-b,\frac{1}{2}]\\
2A_2Q(t),&  t\in [\frac{1}{2},b].
\end{array}\right.
\end{equation}

Now, we shall prove $Q(x)=0$ a.e. on $[0,b]$.

Here we only consider the case $1\ge b\geq\frac{3}{4}$, and the case $\frac{1}{2}<b<\frac{3}{4}$ is similar. From (\ref{F1})
and (\ref{Fxs1}), we see that
\begin{equation*}
2A_2Q(t)+\int_t^bK(x,t)Q(x)dx=0,  \quad  t\in
[\textstyle{\frac{1}{2}},b].
\end{equation*}

Observe from (\ref{yyw2}) that $A_2\neq 0$. Thus, this is a homogeneous
Volterra integral equation, then $Q(t)=0$ a.e. on
$[\frac{1}{2},b]$.

When $t\in [b-\frac{1}{2},\frac{1}{2}]$, then $1-t\in
[\frac{1}{2},\frac{3}{2}-b]$, so $Q(1-t)=0$ for almost all $t\in
[b-\frac{1}{2},\frac{1}{2}]$. This implies, from (\ref{F1}) and (\ref{Fxs1}), that
\begin{equation*}
Q(t)+\int_t^{\frac{1}{2}}K(x,t)Q(x)dx=0,  \quad t\in
\textstyle{[b-\frac{1}{2},\frac{1}{2}]},
\end{equation*}
which implies $Q(t)=0$ a.e. on $[b-\frac{1}{2},\frac{1}{2}]$.

When $t\in [1-b,b-\frac{1}{2}]$, it follows $t+\frac{1}{2},
1-t\in [\frac{3}{2}-b,b]\subset[b-\frac{1}{2},b]$, thus,
$Q(1-t)=Q(t+\frac{1}{2})=0$ for almost all $t\in [1-b,b-\frac{1}{2}]$.
This implies, from (\ref{F1}) and (\ref{Fxs1}), that
\begin{equation*}
Q(t)+\int_t^{b-\frac{1}{2}}K(x,t)Q(x)dx=0,  \quad  t\in
\textstyle{[1-b,b-\frac{1}{2}]},
\end{equation*}
which implies $Q(t)=0$ a.e. on $[1-b,b-\frac{1}{2}]$.

When $t\in [0,1-b]$, it follows $t+\frac{1}{2}\in
[\frac{1}{2},\frac{3}{2}-b]$, so $Q(t+\frac{1}{2})=0$ for almost all
$t\in [0,1-b]$. It follows from (\ref{F1}) and (\ref{Fxs1}) that
\begin{equation*}
Q(t)+\int_t^{1-b}K(x,t)Q(x)dx=0,  \quad  t\in
[0,1-b],
\end{equation*}
which implies $Q(t)=0$ a.e. on $[0,1-b]$.

Therefore, in the case $1\ge b\geq\frac{3}{4}$, $Q(x)=0$ a.e. on
$[0,b]$. In the case $\frac{1}{2}<b<\frac{3}{4}$, the proof is similar. Consequently, $Q(x)=0$ a.e. on $[0,b]$ for $b\in(\frac{1}{2},1]$. The proof is complete.
\end{proof}
In order to prove the main results, we also need the following two lemmas. One can find them in \cite{GS2,LEV}.
\begin{lemma}\lb{l3}
Assume that $E(\lambda)$ is an entire function of order less than one. If $\mathop {\lim }\limits_{ |t|\to \infty ,t \in \mathbb{R}}  E(it)=0$, then $E(\lambda)\equiv 0 $ on the whole complex plane.
\end{lemma}
\begin{lemma}\lb{l2}
For any entire function $g(k)\not\equiv0$ of exponential type, the following inequality holds,
 \begin{equation*}
\mathop {\varliminf }\limits_{r \to \infty }  \frac{n(r)}{r}\leq\frac{1}{2\pi}\int_0^{2\pi}h_g(\theta)d\theta,
 \end{equation*}
where $n(r)$ is the number of zeros of $g(k)$ in the disk $|k|\leq r$ and $h_g(\theta):=\mathop {\varlimsup }\limits_{r \to \infty }\frac{\ln |g(re^{i\theta})|}{r}$ with $k=re^{i\theta}$.
\end{lemma}

\section{Proofs}
This section provides the proofs of Theorem \ref{2.0} and Corollary \ref{2.1}.
The proof of Theorem \ref{3.0} is similar to that of Theorem \ref{2.0},
thus we omit it.

\begin{proof}[Proof of Theorem \ref{2.0}]
We consider two boundary value problems: one is the problem $B(q,h,H,a_1,a_2)$ and the other is
$B(\tilde{q},\tilde{h},\tilde{H},\tilde{a}_1,\tilde{a}_2)$, which produces the same data as in Theorem \ref{2.0}. Now under the corresponding assumptions in Theorem \ref{2.0}, we try to prove $B(q,h,H,a_1,a_2)=B(\tilde{q},\tilde{h},\tilde{H},\tilde{a}_1,\tilde{a}_2)$.

Firstly, from the assumptions of Theorem \ref{2.0} and the formulas (\ref{lbdn})$-$(\ref{xc1}), one can easily obtain that
$a_1=\tilde{a}_1$ and $H=\tilde{H}$.

{(i)} Recall $\lambda=k^2$ and $b\in[\frac{1}{2},1]$, and denote
\begin{equation}\label{yywm1}
\begin{split}
g(k):=
&\int_0^b[\tilde{q}(x)-q(x)]y(x,\lambda)\tilde{y}(x,\lambda)dx+(\tilde{h}-h)\\
&\quad +a_1(\tilde{a}_2-a_2)y(\textstyle{\frac{1}{2}-0},\lambda)\tilde{y}(\textstyle{\frac{1}{2}-0},\lambda).
\end{split}
\end{equation}
Since $q(x)=\tilde{q}(x)$ a.e. on $[b,1]$ then
\begin{equation*}
\begin{split}
g(k)=
&\int_0^1[\tilde{q}(x)-q(x)]y(x,\lambda)\tilde{y}(x,\lambda)dx+(\tilde{h}-h)\\
&\quad +a_1(\tilde{a}_2-a_2)y(\textstyle{\frac{1}{2}-0},\lambda)\tilde{y}(\textstyle{\frac{1}{2}-0},\lambda).
\end{split}
\end{equation*}
Note that for fixed $\lambda$ there holds
 \begin{equation*}\label{xc10}
 \begin{split}
\int_0^1[\tilde{q}(x)-q(x)]y(x,\lambda)\tilde{y}(x,\lambda)dx=
&(\tilde{y}'y-\tilde{y}y')(x,\lambda)|_0^{\frac{1}{2}-0}\\
& +(\tilde{y}'y-\tilde{y}y')(x,\lambda)|_{\frac{1}{2}+0}^1,
\end{split}
\end{equation*}
together with (\ref{yz}) and the initial values of $y(x,\lambda)$ and $\tilde{y}(x,\lambda)$ at $x=0$, we can transform (\ref{yywm1}) into
\begin{equation}\label{xc12}
  g(k)=(\tilde{y}'y-\tilde{y}y')(1,\lambda).
\end{equation}
 It follows from (\ref{btj}) and (\ref{xc12}) with $H=\tilde{H}$ that
\begin{equation}\lb{gg3}
 g(k)=0\quad \text{at}\quad k=\pm \sqrt{\kappa_{n}}\quad\text{for} \quad\kappa_{n}\in S_1\cup S_2.
\end{equation}

From (\ref{yyw11}), (\ref{yyw21}) and (\ref{yywm1}), we see that $g(k)$ is an entire
function of $k$ of exponential type $\leq 2b$, and satisfies
\begin{equation}\lb{hlbd5}
|g(k)|\leq C_0e^{2b|{\rm Im}k|}
\end{equation}
for some positive constant $C_0$.
Since $|{\rm Im}k|=r|\sin\theta|$, where $k=re^{i\theta}$, it follows from (\ref{hlbd5}) that
\begin{equation*}
h_g(\theta):=\mathop {\varlimsup }\limits_{r \to \infty }\frac{\ln |g(re^{i\theta})|}{r}\leq2b|\sin\theta|,
\end{equation*}
which implies
\begin{equation}\label{gg6}
 \frac{1}{2\pi}\int_0^{2\pi}h_g(\theta)d\theta\leq\frac{2b}{2\pi}\int_0^{2\pi}|\sin\theta|d\theta=\frac{4b}{\pi}.
\end{equation}

On the other hand, recalling the definitions of the functions $N_{\Lambda_i}(r)$ ($i=1,2$), and using (\ref{lbdn}) and (\ref{xc}), one gets
\begin{equation}\label{cx4}
 N_{\Lambda_i}(r)=\frac{r}{\pi}[1+o(1)],\quad r\to \infty,\quad i=1,2.
\end{equation}
Let $n(r)$ be the number of zeros of $g(k)$ in the disk $|k| \le r$, then using (\ref{xc2}) and (\ref{cx4}) one obtains
\begin{equation}\label{gg5}
n(r)\ge 2N_{S_1\cup S_2}(r)\ge\frac{4\sigma r}{\pi}[1+o(1)],\quad r\to \infty.
\end{equation}

Using Lemma \ref{l2}, together with (\ref{gg6}) and (\ref{gg5}), we obtain $\sigma\leq b$ if the entire function $g(k)\not\equiv0$. However, now $\sigma>b$, which implies that $g(k)\equiv0$ on the whole complex plane. Therefore, it follows from Lemma \ref{l1} that $q(x)=\tilde{q}(x)$ a.e. on $[0,b]$, $a_2=\tilde{a}_2$ and $h=\tilde{h}$.

{(ii)} Recall $\{\kappa_{i,n}\}_{n\in \mathbb{N}_0}= S_i$, $i=1,2$.
Define
\begin{equation}\label{2g}
G(\lambda):=g(k),\
\Phi(\lambda):=\prod_{n=0}^\infty\left(1-\frac{\lambda}{\kappa_{1,n}}\right)\left(1-\frac{\lambda}{\kappa_{2,n}}\right),\ E(\lambda):=\frac{G(\lambda)}{\Phi(\lambda)}.
\end{equation}
Note that when $\kappa_{1,n}=0$ or $\kappa_{2,n}=0$ the expression $\Phi(\lambda)$ requires a minor modification.
From (\ref{yyw1}), (\ref{yyw2}) and (\ref{yywm1}), we know that $ G(\lambda)$ is an entire function of $\lambda$ of order at most $\frac{1}{2}$. We shall show that the function $ \Phi(\lambda)$ is also an entire function of order at most $\frac{1}{2}$. If it is true, then it follows from (\ref{gg3}) that $ E(\lambda)$ is an entire function of $\lambda$ of order at most $\frac{1}{2}$.

By virtue of (\ref{lbdn}) and (\ref{xc}), we have
\begin{equation}\label{cx5}
  \frac{1}{\kappa_{i,n}}=O\left(\frac{1}{n^2}\right),\quad  i=1,2,\quad n\to\infty,
\end{equation}
which implies that the series
\begin{equation*}
\sum_{n=0}^\infty\left(\left|\frac{\lambda}{\kappa_{1,n}}\right|+\left|\frac{\lambda}{\kappa_{2,n}}\right|\right)
\end{equation*}
converges uniformly on bounded subsets of $\mathbb{C}$.
Therefore, the infinite product $\Phi(\lambda)$ in (\ref{2g}) converges to an entire function of $\lambda$, whose roots are exactly $\kappa_{1,n}$ and $\kappa_{2,n}$, $n\in\mathbb{N}$. Denote
\begin{equation*}
  \rho_\Phi:=\inf\left\{p:\sum_{n=0}^\infty\left(\frac{1}{|\kappa_{1,n}|^p}+\frac{1}{|\kappa_{2,n}|^p}\right)<\infty\right\},
\end{equation*}
which is called \emph{convergence exponent of zeros} of the canonical product of $\Phi$ in (\ref{2g}). Clearly, $\rho_\Phi\le\frac{1}{2}$ by the estimates (\ref{cx5}). Since the order of canonical product of an entire function is equal to its convergence exponent of zeros (see \cite[p.16]{LEV}), thus we conclude that the order of canonical product of $\Phi$ is at most $\frac{1}{2}$. By Hadamard's factorization theorem, the infinite product in (\ref{2g}) is the canonical product of the function $\Phi(\lambda)$, and so the order of $\Phi$ is at most $\frac{1}{2}$.

From Lemma \ref{l3}, we know that it is sufficient to show that
\begin{equation}\label{g6}
\mathop {\lim }\limits_{ |t|\to \infty ,t \in \mathbb{R}}  E(it)=0.
\end{equation}

Denote
\begin{equation}\label{xc15}
 \alpha_n:=\frac{n^2\pi^2}{\sigma_1^2},\quad\beta_n:=\frac{\gamma_n^2}{\sigma_2^2},\quad n\in\mathbb{N}_0,
\end{equation}
 and
\begin{equation}\label{xc14}
\begin{split}
\Phi_0(\lambda):&=\sigma_1^{-1}\sqrt{\lambda}\sin(\sigma_1\sqrt{\lambda})[\cos(\sigma_2\sqrt{\lambda})+a]\\
&=(1+a)\lambda\prod_{n=1}^\infty\left(1-\frac{\lambda}{\alpha_n}\right)\prod_{n=0}^\infty\left(1-\frac{\lambda}{\beta_n}\right).
\end{split}
\end{equation}
Recalling the well known inequalities \cite{FYU1,YUR3}, for sufficiently large real $t$ and some positive constant $M$ independent of $t$, there hold
\begin{equation}\label{xc16}
  |\sin(\sigma_1\sqrt{it})|\ge Me^{\sigma_1|{\rm Im}\sqrt{it}|},\quad |\cos(\sigma_2\sqrt{it})|\ge Me^{\sigma_2|{\rm Im}\sqrt{it}|}.
\end{equation}
Indeed, one can choose
$$M=\min\left\{\inf_{|t|>1}e^{-\sigma_1|{\rm Im}\sqrt{it}|}|\sin(\sigma_1\sqrt{it})|,\inf_{|t|>1}e^{-\sigma_2|{\rm Im}\sqrt{it}|}|\cos(\sigma_2\sqrt{it})|\right\}.$$
Thus, for sufficiently large real $t$,
\begin{equation}\label{xc17}
  |\Phi_0(it)|\ge C_1 \sqrt{|t|}e^{(\sigma_1+\sigma_2)|{\rm Im}\sqrt{it}|}=C_1 \sqrt{|t|}e^{2b|{\rm Im}\sqrt{it}|}.
\end{equation}
Here $C_1$ and the following $C_2,C_3,C_4$ are all some positive constants independent of $t$.
Together with (\ref{hlbd5}), (\ref{2g}), (\ref{xc14}) and (\ref{xc17}), we have
\begin{equation}\label{g7}
\begin{split}
\!\!\!| E(it)|\!=\left|\frac{G(it)}{\Phi_0(it)} \right|\left|\frac{\Phi_0(it)}{\Phi(it)} \right|&\leq C_2\frac{e^{2b|{\mathop{\rm Im}\nolimits} \sqrt {it} |}}{\sqrt{|t|}}\frac{1}{e^{2b|{\mathop{\rm Im}\nolimits} \sqrt {it} |}}\prod_{n=1}^\infty\left|\frac{1-\frac{it}{\alpha_n}}{1-\frac{it}{\kappa_{1,n}}}\right|\left|\frac{1-\frac{it}{{\beta}_n}}{1-\frac{it}{\kappa_{2,n}}}\right|\\
  &\!=\frac{C_2}{\sqrt{|t|}}\prod_{n=1}^\infty\left(\left|\frac{1+\frac{t^2}{\alpha_n^2}}{1+\frac{t^2}{\kappa_{1,n}^2}}\right|
  \left|\frac{1+\frac{t^2}{{\beta}_n^2}}{1+\frac{t^2}{\kappa_{2,n}^2}}\right|\right)^{1/2}.
\end{split}
\end{equation}
Using the inequalities
\begin{equation*}
  \frac{1+c}{1+d}\le\frac{c}{d} \quad \text{if} \quad c\ge d\ge 0;\quad \frac{1+c}{1+d}\le1 \quad \text{if} \quad d\ge c\ge 0,
\end{equation*}
and noting that $\#\{n:\kappa_{1,n}<0 \text{ or } \kappa_{2,n}<0\}<\infty$,
we have
\begin{equation}\label{cx6}
\begin{split}
|E(it)|&\leq\frac{C_3}{\sqrt{|t|}}\prod_{\kappa_{1,n}\ge\alpha_n}\frac{\kappa_{1,n}}{\alpha_n}
\prod_{\kappa_{2,n}\ge\beta_n}\frac{\kappa_{2,n}}{\beta_n}\\
&\le\frac{C_3}{\sqrt{|t|}}\prod_{n=0}^\infty\left[1+\frac{(\kappa_{1,n}-\alpha_n)_+}{\alpha_n}\right]\prod_{n=0}^\infty\left[1+\frac{(\kappa_{2,n}-\beta_n)_+}{\beta_n}\right].
\end{split}
\end{equation}
It follows from (\ref{cx0}) and (\ref{xc15}) that for sufficiently large $|t|$,
\begin{equation*}
   | E(it)|\leq\frac{C_4}{\sqrt{|t|}}.
\end{equation*}
This implies that the assertion (\ref{g6}) holds. We have finished the proof.
\end{proof}
\begin{proof}[Proof of Corollary \ref{2.1}]
We agree that $\sigma_2=0$ means that $S_2=\emptyset$.
Let $b=\frac{1}{2}$ in (\ref{yywm1}), and set $\sigma_1=1,\sigma_2=0$, $a_1=\tilde{a}_1$ and $H=\tilde{H}$ in the part {(ii)} of the proof in Theorem \ref{2.0} to obtain the result: $q(x)=\tilde{q}(x)$ a.e. on $[0,\frac{1}{2}]$, $a_2=\tilde{a}_2$ and $h=\tilde{h}$.
\end{proof}

\noindent {\bf Acknowledgments.} The authors would like to thank the referees for valuable suggestions and comments.
The author Xu was supported by Innovation Program for Graduate Students of Jiangsu Province of China (KYLX16\_0412). The authors Xu and Yang were supported in part by the National Natural Science Foundation of
China (11171152 and 91538108) and Natural Science Foundation of Jiangsu Province of China (BK 20141392).

\end{document}